\documentclass{article}
\usepackage{latexsym}
\usepackage{amsmath, amsfonts, amsthm, mathabx}
\usepackage{amsmath, amsfonts, amsthm}
\usepackage{epsfig}
\usepackage{stmaryrd}
\usepackage{multicol}
\usepackage{pdfsync}
\usepackage{dsfont}
\usepackage{comment}
\usepackage{algorithmic}
\usepackage{algorithm}
\usepackage{caption}
\usepackage{psfrag}
\usepackage{subfig}
\usepackage{xcolor}
\usepackage{graphics}

\usepackage{mathtools}

\newcommand{\norm}[1]{\left\lVert #1 \right\rVert}

\newtheorem{theorem}{Theorem}[section]

\newtheorem{lemma}[theorem]{Lemma}

\theoremstyle{definition}

\newtheorem{assumption}[theorem]{Assumption}

\theoremstyle{remark}
\newtheorem{remark}[theorem]{Remark}

\numberwithin{equation}{section}

\title{Online Adjoint Methods for Optimization of PDEs}
\author{Justin Sirignano\footnote{Mathematical Institute, University of Oxford, E-mail: Justin.Sirignano@maths.ox.ac.uk} \phantom{.}  and Konstantinos Spiliopoulos\footnote{Department of Mathematics and Statistics, Boston University, Boston, E-mail: kspiliop@math.bu.edu}
\thanks{K.S. was partially supported by the National Science Foundation (DMS 1550918, DMS 2107856) and  Simons Foundation Award  672441}\\
}

\date{\today}

\begin{document}

\maketitle

\begin{abstract}
We present and mathematically analyze an online adjoint algorithm for the optimization of partial differential equations (PDEs). Traditional adjoint algorithms would typically solve a new adjoint PDE at each optimization iteration, which can be computationally costly. In contrast, an online adjoint algorithm updates the design variables in continuous-time and thus constantly makes progress towards minimizing the objective function. The online adjoint algorithm we consider is similar in spirit to the the pseudo-time-stepping, one-shot method which has been previously proposed. Motivated by the application of such methods to engineering problems, we mathematically study the convergence of the online adjoint algorithm. The online adjoint algorithm relies upon a time-relaxed adjoint PDE which provides an estimate of the direction of steepest descent. The algorithm updates this estimate continuously in time, and it asymptotically converges to the exact direction of steepest descent as $t \rightarrow \infty$. We rigorously prove that the online adjoint algorithm converges to a critical point of the objective function for optimizing the PDE. Under appropriate technical conditions, we also prove a convergence rate for the algorithm. A crucial step in the convergence proof is a multi-scale analysis of the coupled system for the forward PDE, adjoint PDE, and the gradient descent ODE for the design variables.
\end{abstract}
\section{Introduction}

Adjoint methods have been widely-used for the optimization of partial differential equations (PDEs), and especially for optimizing PDEs modeling engineering systems. Examples include \cite{Bueno}, \cite{Cagnetti}, \cite{Giles3}, \cite{Giles0}, \cite{Giles1}, \cite{Giles2}, \cite{Jameson4}, \cite{Jameson6}, \cite{Jameson2}, \cite{Jameson3}, \cite{Jameson5},  \cite{Giles4}, \cite{Protas}, \cite{Jameson1}, \cite{Knopoff2013}, \cite{Brandenburg2009}, and \cite{Hinze2009}.

Traditional adjoint algorithms consist of an iteration where at each iteration a new adjoint PDE must be solved to calculate the gradient descent step. During the course of optimization, many adjoint PDEs must be solved, which in certain cases can be computationally costly. As an alternative, time-stepping and pseudo-time-stepping methods (often in combination with one-shot methods) have been proposed, where one views time-independent PDEs as stationary states of appropriate dynamical systems and studies the behavior of the latter in the long-time regime, i.e., after their transient phase. We refer the interested reader to \cite{Bosse2014,GaugerHambi2012,Gunther2016,Hazra2007,HazraSchultz2004,Kaland2014,Taasan1991,Taasan1995}
and the references therein for certainly a non-exhaustive list of representative references.

In this paper, we  couple a time-relaxed adjoint PDE with a continuous-time update equation for the variables that are being optimized. The time-relaxed adjoint PDE yields an estimate of the direction of steepest descent, and updates this estimate \emph{continuously} in time. The optimization variables are also updated continuously in time using this online estimate of the direction of steepest descent. The focus of our paper is the mathematical analysis of this ``online adjoint algorithm". As $t \rightarrow \infty$, the solution of the time-relaxed adjoint PDE asymptotically matches the exact direction of steepest descent. A crucial step in the convergence proof is a multi-scale analysis of the coupled system for the forward PDE, adjoint PDE, and the gradient descent ODE for the design variables.

We prove convergence and convergence rates for the online adjoint algorithm for a certain class of PDEs. Specifically, in our theoretical analysis, we consider the optimization problem where we seek to minimize the objective function:
\begin{eqnarray}
J(\theta) = \frac{1}{2} \int_{U} \bigg{(} u^{\ast}(x) - h(x) \bigg{)}^2 dx + \frac{\gamma}{2} \norm{\theta}^2_2,
\label{ObjectivePDEIntro}
\end{eqnarray}
where $h$ is a target profile.  $ \gamma \norm{\theta}^2_2$ is a regularization term where $\gamma > 0$ and $\norm{ \cdot}_2$ is the $\ell_2$ norm. $u^{\ast}$ satisfies the elliptic PDE
\begin{align}
A u^{\ast}(x) &= f( x, \theta ), \quad x\in U\nonumber\\
u^{\ast}(x)&= 0,\quad x\in \partial U,
\label{EllipticPDEIntro}
\end{align}
where $A$ is a standard second-order elliptic operator.
Thus, we wish to select a parameter $\theta$ such that the solution $u^{\ast}$ of the PDE (\ref{EllipticPDEIntro}) is as close as possible to the target profile $h$.

If $A^{\dagger}$ denotes the formal adjoint operator to $A$, then the adjoint PDE is
\begin{align}
A^{\dagger} \hat u^{\ast}(x)&= u^{\ast}(x) - h(x), \quad x\in U\nonumber\\
\hat u^{\ast}(x)&= 0,\quad x\in \partial U
\label{AdjointPDEIntro}
\end{align}

The gradient of the objective function (\ref{ObjectivePDEIntro}) can be evaluated using the solution $\hat u^{\ast}$ to the adjoint PDE (\ref{AdjointPDEIntro}). By Lemma \ref{Eq:ObjFcnRep} we have that
\begin{eqnarray}
\nabla_{\theta} J(\theta) = \int_{U} \hat u^{\ast}(x) \nabla_{\theta} f(x, \theta) dx + \gamma \theta.
\label{GradIntro}
\end{eqnarray}

Thus, the adjoint PDE (\ref{AdjointPDEIntro}) can be used to evaluate the gradient of the objective function, which in turn can be used to optimize over the PDE (\ref{EllipticPDEIntro}). A key advantage of adjoint methods is that, no matter how large the dimension of $\theta$ is, the adjoint PDE (\ref{AdjointPDEIntro}) is the same dimension as the original PDE (\ref{EllipticPDEIntro}).

\subsection{The online adjoint algorithm}
The online adjoint algorithm  optimizes the objective function $J(\theta)$ via a continuous-time equation for the update of the parameter $\theta(t)$; see also \cite{HazraSchultz2004,Taasan1995} for related formulations.
The direction of steepest descent is estimated using a time-relaxation of the adjoint PDE. The estimate and the optimization variables are both simultaneously updated continuously in time. An appropriately chosen learning rate parameter is introduced, which allows to guarantee both well posedness of the algorithm for all times (Theorem \ref{T:WellPosednessThm}) and convergence as $t\rightarrow\infty$ (Theorems \ref{ConvergenceTheorem} and \ref{T:ConvergenceRate}). 

The online adjoint algorithm satisfies the equations:
\begin{align}
\frac{\partial u}{\partial t}(t,x) &= -A u(t,x) + f(x,\theta(t)), \quad x\in U, \phantom{.} t>0 \notag \\
\frac{\partial \hat u}{dt}(t,x) &= -A^{\dagger} \hat u(t,x)  + ( u(t,x) - h(x)), \quad x\in U, \phantom{.} t>0 \notag \\
\frac{d \theta}{dt}(t) &= - \alpha(t) \bigg{(} \int_{U} \hat u(t,x) \nabla_{\theta} f (x,  \theta(t) ) dx + \gamma \theta(t) \bigg{)}\nonumber\\
u(t,x)&=\hat{u}(t,x)=0, \quad x\in\partial U, \phantom{.} t>0\nonumber\\
u(0,x)&=u_{0}(x), \hat{u}(0,x)=\hat{u}_{0}(x),
\label{TimeRIntro}
\end{align}
where $\alpha(t)$ is an appropriately chosen learning rate. The PDEs for $u$ and $\hat u$ can be viewed as time relaxations of the PDE (\ref{EllipticPDEIntro}) and its adjoint PDE (\ref{AdjointPDEIntro}). It is easy to see that $\int_{U} \hat u(t,x) \nabla_{\theta} f (x,  \theta(t) ) dx$ is an estimate for the direction of steepest descent $\int_{U} \hat u^{\ast}(x) \nabla_{\theta} f (x,  \theta(t) ) dx$.

Apart from the generic formulation of the online adjoint algorithm as presented in (\ref{TimeRIntro}), our main contribution is two-fold. First, we prove that as $t\rightarrow\infty$, $\norm{\nabla J(\theta(t))}\rightarrow 0$. Namely, we prove that $\theta(t)$ converges to a stationary point of $J(\theta)$. We emphasize here that in order to do so, no assumptions on convexity of $J$ are needed. Secondly, if we further assume that $J(\theta)$ is strongly convex, then  we also prove a convergence rate of $\theta(t)$ to the global minimum of $J(\theta)$.

In practice, the online adjoint algorithm (\ref{TimeRIntro}) is implemented by simultaneously solving the coupled ODE-PDE system using numerical methods such as finite-difference methods. Either explicit or implicit finite difference methods can be used. For example, an explicit finite difference method for implementing (\ref{TimeRIntro}) would be:
\begin{itemize}
\item Update $u$ and $\hat u$:
\begin{eqnarray}
u(t + \Delta, x) &=& u(t,x) + \bigg{(} - A u(t,x) + f(x,\theta(t)) \bigg{)} \Delta, \notag \\
\hat u(t + \Delta,x) &=& \hat u(t,x) + \bigg{(} -A^{\dagger} \hat u(t,x)  + ( u(t,x) - h(x)) \bigg{)} \Delta,
\end{eqnarray}
where $\Delta$ is the time-step size.
\item Then, update the parameter $\theta$:
\begin{eqnarray}
\theta(t+\Delta) = \theta(t) - \alpha(t) \bigg{(}  \int_{U} \hat u(t,x) \nabla_{\theta} f (x,  \theta(t) ) dx + \gamma \theta(t) \bigg{)} \Delta.
\end{eqnarray}
The spatial domain $x$ is discretized and a finite-difference method is used to approximate the operator $A$. The integrals are discretized as appropriate sums.
\end{itemize}

The focus of our paper is to rigorously prove the convergence of the online adjoint algorithm for linear elliptic PDEs. In practice, real-world applications will typically require optimizing over nonlinear PDEs. The online adjoint algorithm can also be used to optimize over nonlinear PDEs. \cite{SirignanoMacArtSpiliopoulos2021} and \cite{SirignanoMacArtPanesi} optimize over the Navier-Stokes equation using our online adjoint algorithm. Numerical optimization with pseudo-time-stepping adjoint methods has also been studied in \cite{Bosse2014,GaugerHambi2012,Gunther2016,Hazra2007,HazraSchultz2004,Kaland2014,Taasan1991,Taasan1995}.

We demonstrate the online adjoint method below for a simple example of a nonlinear PDE. Consider the equation
\begin{eqnarray}
0 &=& - \theta_1 u \frac{\partial u}{\partial x} - \theta_2 u \frac{\partial u}{\partial y} + \frac{\partial^2 u}{\partial x^2} + \frac{\partial^2 u}{\partial y^2},
\label{Burgers}
\end{eqnarray}
where $(x,y) \in [0,1] \times [0,1]$ and with boundary conditions $u(0, y) = 1$, $u(1,y) = -1$, $u(x, 0) = 1$, and $u(x,1) = -1$. The parameters to be optimized over are $\theta = (\theta_1, \theta_2)$ and the objective function is (\ref{ObjectivePDEIntro}) with $\gamma = 0$. The target function $h$ is the solution to (\ref{Burgers}) with $\theta = (10, 10)$. That is, our goal is to solve the inverse problem of recovering the parameters in the PDE (\ref{Burgers}) given an observed solution.

The adjoint PDE for (\ref{Burgers}) is
\begin{eqnarray}
0&=&   \theta_1  u \frac{\partial \hat u }{\partial x}    + \theta_2  u \frac{\partial \hat u }{\partial y}  + \frac{\partial^2 \hat u}{\partial x^2} + \frac{\partial^2 \hat u}{\partial y^2} + (u - h),
\label{BurgersAdjointPDE}
\end{eqnarray}
where $(x,y) \in [0,1] \times [0,1]$ and with boundary conditions $\hat u(0, y) = \hat u(1,y) = \hat u(x, 0) = \hat u(x,1) = 0$. The gradient of the objective function is given by the formula
\begin{eqnarray}
\frac{\partial J(\theta) }{\partial \theta_1} &=& - \int_{0}^1 \int_0^1  \hat u u \frac{\partial u}{\partial x} (x,y) dx dy, \notag \\
\frac{\partial J(\theta) }{\partial \theta_2} &=& - \int_{0}^1 \int_0^1  \hat u u \frac{\partial u}{\partial y} (x,y) dx dy.\nonumber
\end{eqnarray}
The online adjoint algorithm can be used to minimize the objective function $J(\theta)$. In our numerical experiment, we use an explicit finite difference method for the numerical solution of the time-relaxed PDE and the time-relaxed adjoint PDE. The PDE variables are updated using an Euler scheme. (Although not implemented here, it is worthwhile noting that higher-order accuracy in time could be achieved with a Runge-Kutta scheme.) Uniform mesh sizes for both time and space are chosen. The PDE operator is approximated using a second-order accurate finite-difference method. The parameter ODEs are also solved using an explicit Euler scheme on the same uniform time grid and the spatial integrals are also discretized as sums using the uniform spatial grid. Figure \ref{OAFig2} demonstrates that the online adjoint algorithm converges to the correct value for the parameters $\theta$ as $t \rightarrow \infty$. The right display in Figure \ref{OAFig2} presents the numerical convergence rate, which satisfies the theoretical convergence rate of $t^{-\frac {1}{2}}$ which we prove for strongly convex objective functions for linear elliptic PDEs in this paper. (In fact, for this specific example, the numerical convergence rate turns out to be faster than $t^{- \frac{1}{2}}$.)
\begin{figure}[!h]
  \centering
   \includegraphics[width=0.4\textwidth]{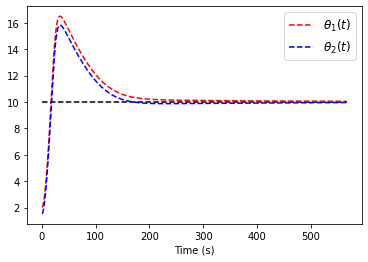}
      \includegraphics[width=0.4\textwidth]{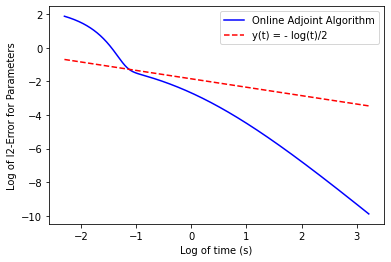}
 \caption{Left: solution for $\theta$ using the online adjoint algorithm versus computational time. Right: numerical convergence rate for $\norm{ \theta - \theta^{\ast}}_2$ where $\theta^{\ast} = (10, 10)$. }
\label{OAFig2}
\end{figure}

\subsection{Organization of the Proof}

The rest of the paper is organized as follows. In Section \ref{S:WellPosedness} we state our assumptions, present in more details the online adjoint algorithm and prove its well posedness in Theorem \ref{T:WellPosednessThm}. Convergence of $\theta(t)$ to a stationary point of $J(\theta)$ is proven in Section \ref{S:ConvergenceStationaryPoint}, Theorem \ref{ConvergenceTheorem}. We emphasize that no convexity requirements on $J(\theta)$ are needed in order to prove convergence. If in addition, we assume that $J(\theta)$ is strongly convex with a single stationary point, then one can prove a convergence rate, Theorem \ref{T:ConvergenceRate}. The latter is the content of Section \ref{S:ConvergenceRate}. %Section \ref{S:NumericalExamples} includes a few numerical studies to illustrate the theoretical results.

\section{Assumptions, notation and well posedness of the online adjoint algorithm} \label{S:WellPosedness}

Let $U$ be an open, bounded subset of $\mathbb{R}^{n}$. We will denote by $(\cdot,\cdot)$ the usual inner product in $H=L^{2}(U)$. We shall assume that the operator $A$ is uniformly elliptic, diagonalizable and dissipative, per Assumption \ref{A:Assumption0}.
\begin{assumption} \label{A:Assumption0}
The operator $A$ is uniformly elliptic. We also assume that $A$ is diagonalizable and dissipative. Namely there exists a countable complete orthonormal basis $\{e_{n}\}_{n\in\mathbb{N}}\subset H$ that consists of  eigenvectors of $A$ corresponding to a non-negative sequence $\{\lambda_{n}\}_{n\in\mathbb{N}}$ of eigenvalues such that
\[
(-A)e_{n}=-\lambda_{n}e_{n}, \quad n\in\mathbb{N}
\]
such that the dissipativity condition $\lambda= \displaystyle \inf_{n\in\mathbb{N}} \lambda_{n}>0$ holds.
\end{assumption}

An example of $A$ is the second order elliptic operator, for definiteness taken to be in divergence form,
\begin{align*}
Au(x)=-\sum_{i,j=1}^{n}\left(a^{i,j}(x)u_{x_{i}}(x)\right)_{x_{j}}+\sum_{i=1}^{n}b^{i}(x)u_{x_{i}}(x)+c(x)u(x).
\end{align*}

The diagonalizable condition of Assumption \ref{A:Assumption0} is automatically satisfied for example by self-adjoint operators, see \cite[Theorem 8.8.37]{GilbargTrudinger}

Before proceeding with the well-posedness of the online adjoint algorithm, let us recall a few basic results that will be useful for the analysis that follows.

Taking the domain of $A$ to be $D(A)=H^{1}_{0}(U)\cap H^{2}(U)$, we have that it is dense in $H=L^{2}(U)$. Then, due to Assumption \ref{A:Assumption0}, elliptic regularity theory gives that the operator A is closed and thus by Hille-Yoshida theorem $(-A)$ is the generator of an analytic strongly contraction semigroup $\{S(t)\}_{t\geq 0}$ on $H$. The spectral assumption made in Assumption \ref{A:Assumption0}  guarantees that
\begin{align}
\norm{S(t)u}_{H}&\leq e^{-\lambda t}\norm{u}_{H}.\label{Eq: ExponentialStability}
\end{align}

The latter also means that  $A$ is a coercive operator. In particular, we will heavily use the fact that for $u\in D(A)$, we have
\begin{align}
( A u, u)&\geq\lambda\norm{u}^{2}_{H}. \label{Eq:CoerciveA}
\end{align}

Notice now that because we are dealing with a real Hilbert space and because $H^{1}_{0}(U)\cap H^{2}(U)$ is dense in $L^{2}(U)$, we obtain that $A^{\dagger}$ is also a coercive operator. Indeed, by definition of the adjoint operator  $A^{\dagger}$ we shall have that for $u\in H^{1}_{0}(U)\cap H^{2}(U)$
\begin{align}
( A^{\dagger} u, u ) &= (  u, A u ) \geq \lambda ( u, u ).\label{Eq:CoercivityAdjoint}
\end{align}

Notice now that under our assumptions the adjoint operator $(-A^{\dagger})$ will also generated an analytic strongly continuous semigroup $\{S^{\dagger}(t)\}_{t\geq 0}$ on $L^{2}(U)$. In particular (\ref{Eq:CoercivityAdjoint}) implies that the adjoint semigroup $S^{\dagger}(t)$ will also be exponential stable. Indeed, by definition we have for $u\in H^{1}_{0}(U)\cap H^{2}(U)$
\begin{align}
\frac{d}{dt}\norm{S^{\dagger}(t)u}^{2}_{H}&= -2( A^{\dagger} S^{\dagger}(t)u,S^{\dagger}(t)u)\leq -2\lambda \norm{S^{\dagger}(t)u}^{2}_{H},
\end{align}
which then due to Gronwall lemma gives
\begin{align}
\norm{S^{\dagger}(t)u}_{H}&\leq e^{-\lambda t}\norm{u}_{H},\label{Eq:ExponentiaStableAdjoint}
\end{align}
proving the exponential stability of $S^{\dagger}(t)$.

Then, if we assume that  $f\in L^{2}(U)$,  classical Lax-Milgram theorem (see for example Chapter 5.8 of \cite{GilbargTrudinger}) says that the elliptic boundary-value problem
\begin{align}
A u^{\ast}(x) &= f( x ), \quad x\in U\nonumber\\
u^{\ast}(x)&= 0,\quad x\in \partial U
\label{EllipticPDE}
\end{align}
has a unique weak solution $u^{\ast}\in H^{1}_{0}(U)$.  The same conclusion will also be true for the adjoint problem governed by the adjoint operator $A^{\dagger}$.

\begin{remark}\label{R:GlobalRegularity}
We also recall here that by classical elliptic regularity results if for given $m\in\mathbb{N}$, $a^{i,j},b^{i},c\in\mathcal{C}^{m+1}(\bar{U})$ with $i,j=1,\cdots,n$ and $\partial U\in\mathcal{C}^{m+2}$, then the unique solution $u^{\ast}$ to (\ref{EllipticPDE}) is such that $u^{\ast}\in H^{m+2}(U)$. Clearly if $a^{i,j},b^{i},c\in\mathcal{C}^{\infty}(\bar{U})$ with $i,j=1,\cdots,n$ and $\partial U\in\mathcal{C}^{\infty}$, then $u^{\ast}\in \mathcal{C}^{\infty}(\bar{U})$.  We refer the interested reader to classical manuscripts, e.g., \cite[Chapter 8]{GilbargTrudinger}, for more details.
\end{remark}

\begin{remark}\label{R:NonZeroBoundaryData}
For notational convenience and without loss of generality,  we have assumed  zero boundary conditions for the PDE (\ref{EllipticPDE}). We can consider the PDE (\ref{EllipticPDE}) with non-zero boundary data, say $u^{\ast}(x)= g(x), x\in \partial U$, under the assumption $g\in H^{1}(U)$ for unique solvability of the corresponding PDE (\ref{EllipticPDE}). 
 We refer the interested reader to classical manuscripts, e.g., \cite[Chapter 8]{GilbargTrudinger}, for more details.
\end{remark}

As briefly presented in the introduction now, let $f(\cdot,\cdot):\mathbb{R}^{n}\times\mathbb{R}^{d}\mapsto \mathbb{R}$ be such that for every $\theta\in\mathbb{R}^{d}$ $f(\cdot,\theta)\in L^{2}(U)$. As with (\ref{EllipticPDE}) the linear PDE
\begin{align}
A u^{\ast}(x) &= f( x,\theta ), \quad x\in U\nonumber\\
u^{\ast}(x)&= 0,\quad x\in \partial U
\label{ForwardPDE}
\end{align}
will have, for each given $\theta\in\mathbb{R}^{d}$, a unique weak solution $u^{\ast}\in H^{1}_{0}(U)$. We shall write $u^{\ast}(x;\theta)$ when we want to emphasize the dependence on $\theta$.

For a given target profile $h\in L^{2}(U)$, the goal is to select $\theta$ to minimize the objective function
\begin{eqnarray}
J(\theta) = \frac{1}{2} ( u^{\ast} - h, u^{\ast} - h ) + \frac{\gamma}{2} \norm{\theta}^2_2,
\label{ObjectiveODE}
\end{eqnarray}
where  $\gamma > 0$.  $ \gamma \norm{\theta}^2_2$ is a regularization term and $\norm{ \cdot}_2$ is the $\ell_2$ norm.

The adjoint PDE satisfies
\begin{align}
A^{\dagger} \hat u^{\ast}(x)&= u^{\ast}(x) - h(x), \quad x\in U\nonumber\\
\hat u^{\ast}(x)&= 0,\quad x\in \partial U
\label{AdjointODE}
\end{align}
and as with (\ref{EllipticPDE}), Assumption \ref{A:Assumption0} and the fact that $u^{\ast} - h\in L^{2}(U)$ guarantee that (\ref{AdjointODE})  has a unique weak solution $\hat u^{\ast}\in H^{1}_{0}(U)$. Notice that since $u^{\ast}$ depends on $\theta$, the same will also be true for $\hat u^{\ast}$ and we shall write $\hat u^{\ast}(x;\theta)$ when we want to emphasize that. The following lemma provides a useful representation for $\nabla_{\theta} J(\theta)$, which will also motivate the form of the online adjoint algorithm.
\begin{lemma}\label{Eq:ObjFcnRep}
Let Assumption \ref{A:Assumption0} and assume that  $h\in L^{2}(U)$. Then, we can write
\begin{eqnarray}
\nabla_{\theta} J(\theta) = (  \hat u^{\ast} , \nabla_{\theta} f(\theta)) + \gamma \theta.
\label{GradientObj}
\end{eqnarray}
\end{lemma}

The proof of Lemma \ref{Eq:ObjFcnRep} is presented at the end of this section. In terms of the learning rate $\alpha(t)$ we make the following assumption.
\begin{assumption}\label{A:LearningRateAssumption}
We assume that the learning rate $\alpha(t)$ is such that $\displaystyle \lim_{t\rightarrow\infty}\alpha(t)=0$ and
\begin{itemize}
\item $\int_{0}^{\infty}\alpha(s)ds=\infty$ and $\int_{0}^{\infty}\alpha^{2}(s)ds<\infty$.
\item $\displaystyle \sup_{t\geq 0} \int_{0}^{t}\alpha(s)e^{-\gamma\int_{s}^{t}\alpha(r)dr}ds<\infty$ and $\displaystyle \lim_{t\rightarrow\infty}\frac{\alpha'(t)}{\alpha(t)}=0$.
\end{itemize}
\end{assumption}

The first part of Assumption \ref{A:LearningRateAssumption} on the learning rate is classical and is also the same used in discrete time algorithms, see for example classical references such as \cite{Benveniste,KushnerYin}. The second part of Assumption \ref{A:LearningRateAssumption} comes up while proving that $\norm{\theta(t)}$ stays bounded for all times and later on in the convergence proof of $\theta(t)$ to a stationary point of $J$. An example of a learning rate that satisfies both parts of Assumption \ref{A:LearningRateAssumption} is $\alpha(t)=\frac{1}{1+t}$.

In terms of the parametric model $f(\cdot,\cdot)$   we make the following assumption
\begin{assumption} \label{AssumptionLinear1}
We assume the following conditions:
\begin{itemize}
\item For each fixed $\theta\in\mathbb{R}^{d}$, $f(\cdot, \theta)$, $\nabla_{\theta} f(\cdot, \theta)$ and $\nabla^{2}_{\theta} f(\cdot, \theta)$ are in $L^{2}(U)$. For each fixed $x\in\mathbb{R}^{n}$, $f(x,\cdot)$, $\nabla_{\theta} f(x, \cdot)$ and $\nabla^{2}_{\theta} f(x, \cdot)$ are bounded. In other words we assume that there exists $C<\infty$ such that
    \[
    \sup_{\theta\in\mathbb{R}^{d}}\left(\norm{f(\theta)}_{L^{2}(U)}+\norm{\nabla_{\theta}f(\theta) }_{L^{2}(U)}+\norm{\nabla^{2}_{\theta}f(\theta) }_{L^{2}(U)}\right)\leq C.
    \]
\item For each $x\in U$, $f(x,\cdot)$ is globally Lipschitz $L^{2}$ Lipschitz constant in $x$.
\end{itemize}
\end{assumption}

In terms of the associated cost function $J(\cdot)$ we make the following Assumption \ref{AssumptionLinear2}.
\begin{assumption} \label{AssumptionLinear2}
We assume that $J(\cdot) \in C^2$ and globally Lipschitz.
\end{assumption}

We emphasize that Assumption \ref{AssumptionLinear2} does not impose any convexity type of assumptions on $J(\theta)$. The online adjoint algorithm satisfies the time-dependent PDEs
\begin{align}
\frac{\partial u}{\partial t}(t,x) &= -A u(t,x) + f(x,\theta(t)), \quad x\in U, t>0 \notag \\
\frac{\partial \hat u}{dt}(t,x) &= -A^{\dagger} \hat u(t,x)  + ( u(t,x) - h(x)), \quad x\in U, t>0 \notag \\
\frac{d \theta}{dt}(t) &= - \alpha(t) \bigg{(} ( \nabla_{\theta} f ( \theta(t) ), \hat u(t) ) + \gamma \theta(t) \bigg{)}\nonumber\\
u(t,x)&=\hat{u}(t,x)=0, \quad x\in\partial U, t>0\nonumber\\
u(0,x)&=u_{0}(x), \hat{u}(0,x)=\hat{u}_{0}(x)
\label{TimeR}
\end{align}

Notice that (\ref{TimeR}) is a non-local coupled system of PDEs. Theorem \ref{T:WellPosednessThm} is about the well-posedness of system (\ref{TimeR}). Also, with slight abuse of notation, with $\theta=\theta(t)$ from (\ref{TimeR}) we shall denote the solutions to (\ref{ForwardPDE}) and (\ref{AdjointODE}), by $u^{\ast}(t,x)$ and $\hat u^{\ast}(t,x)$ respectively.

\begin{theorem} \label{T:WellPosednessThm}
Assume Assumptions \ref{A:Assumption0}, \ref{A:LearningRateAssumption} and \ref{AssumptionLinear1} and that $u_{0},
\hat{u}_{0}, h\in L^{2}(U)$. There exists a unique mild solution $u,\hat{u}\in \mathcal{C}((0,\infty); W^{2,2}_{0}(U))\cap \mathcal{C}^{1}((0,\infty);L^{2}(U))$ and $\theta\in\mathcal{C}^{1}((0,\infty))$ to equation (\ref{TimeR}). If in addition $h\in L^{\infty}$, then $u,\hat{u}\in \mathcal{C}((0,\infty); W^{2,p}_{0}(U))\cap \mathcal{C}^{1}((0,\infty);L^{p}(U))$ for any $p\geq 2$ and if further we assume that $u_{0},\hat{u}_{0}\in W^{2,p}(U)$, then $u,\hat{u}\in \mathcal{C}([0,\infty); W^{2,p}_{0}(U))\cap \mathcal{C}^{1}([0,\infty);L^{p}(U))$ for any $p\geq 2$. In addition, we have that there exists some constant $K<\infty$ such that
\begin{eqnarray}
\sup_{t \geq 0} \left[ \norm{u(t)}_{L^{2}(U)}+\norm{\hat{u}(t)}_{L^{2}(U)}+\norm{\theta(t)}_2 \right]< K.
\label{UnitormThetaBound}
\end{eqnarray}
\end{theorem}

\begin{remark}
At this point we mention that even though in Assumption \ref{AssumptionLinear1} we have assumed that $\norm{f(\theta)}_{L^{2}(U)}$ is uniformly bounded, an investigation of  the proof of Theorem \ref{T:WellPosednessThm} shows that this assumption can be relaxed. In particular, at the expense of slightly more elaborate estimates, the results of this paper (which heavily rely on (\ref{UnitormThetaBound}) being true) hold if we assume instead
that $\norm{f(\theta)}_{L^{2}(U)}$ grows linearly in $\norm{\theta}_{\ell_{2}}$ with bounded derivatives, i.e.,
$\norm{f(\theta)}_{L^{2}}\leq C(1+\norm{\theta}_{\ell_{2}})$, still with $\left(\norm{\nabla_{\theta}f(\theta)}_{L^{2}}+\norm{\nabla^{2}_{\theta}f(\theta) }_{L^{2}(U)}\right)<C$ and additionally that the regularization coefficient $\gamma>0$ is large enough depending on the $L^{2}$ norms of $u_{0}(x)$ and $\hat{u}_{0}(x)$. We have chosen to present the results for uniformly bounded $\norm{f(\theta)}_{L^{2}(U)}$ for presentation purposes and because in this case we do not need any additional restriction on the magnitude of $\gamma$ other than being strictly positive.
\end{remark}

Let us conclude this section with the proofs of Lemma \ref{Eq:ObjFcnRep} and Theorem \ref{T:WellPosednessThm}.
\begin{proof}[Proof of Lemma \ref{Eq:ObjFcnRep}]
(\ref{GradientObj}) can be derived using the definition of the adjoint PDE (\ref{AdjointODE}). Define $\tilde u = \nabla_{\theta} u^{\ast}$. Differentiating (\ref{ForwardPDE}) yields
\begin{align}
A \tilde u(x) &= \nabla_{\theta} f( x,\theta ), \quad x\in U\nonumber\\
\tilde u(x)&= 0,\quad x\in \partial U\nonumber
\end{align}

Integration by parts yields
\begin{eqnarray*}
( \hat u^{\ast}, A \tilde u ) &=&  ( A^{\dagger} \hat u^{\ast}, \tilde u ).
\end{eqnarray*}

Due to (\ref{ForwardPDE}), this yields the equation
\begin{eqnarray*}
( A^{\dagger} \hat u^{\ast}, \tilde u ) &=& ( \hat u^{\ast}, \nabla_{\theta} f(\theta) ).
\end{eqnarray*}

Using the definition of the adjoint PDE (\ref{AdjointODE}),
\begin{eqnarray*}
(u^{\ast} - h, \tilde u) &=& ( \hat u^{\ast}, \nabla_{\theta} f(\theta) ).
\end{eqnarray*}

Recalling the objective function (\ref{ObjectiveODE}), we then write
\begin{eqnarray*}
\nabla_{\theta} J(\theta) &=& ( u^{\ast} - h, \tilde u ) + \gamma \theta \notag \\
&=& ( \hat u^{\ast}, \nabla_{\theta} f(\theta) ) + \gamma \theta,
\end{eqnarray*}
which yields (\ref{GradientObj}).
\end{proof}

\begin{proof}[Proof of Theorem \ref{T:WellPosednessThm}]
Let us define the index set $\mathcal{G}=\{1,2,3\}$ and the space $\Theta=U\times \mathcal{G}$. Define the variable $y=(x,\zeta)\in \Theta$ and the measure $dn=dx\otimes d\iota$ on $\Theta$ where $d\iota$ denotes the counting measure on $\mathcal{G}$. Define now the Banach space $X^{2}=L^{2}(\Theta,dn)$. Similarly we denote by $H^{2}(\Theta)=W^{2,2}(\Theta)$ the Banach space of functions $f$ on $\Theta$ such that for each $\zeta\in \mathcal{G}$ we have $f(\cdot,\zeta)\in H^{2}(U)$ with norm $\norm{f}_{H^{2}(\Theta)}=\sum_{\zeta=1}^{3}\norm{f(\cdot,\zeta)}_{H^{2}(U)}$.

Setting $v(t,x)=(u(t,x),\hat{u}(t,x),\theta(t))$ and $\rho(t,y)=v_{\zeta}(t,x)$ (the $\zeta'$th component of the vector $v$) for $y=(x,\zeta)\in U\times\mathcal{G}$ we consider the evolution equation on $\Theta$ given by
\begin{align}
\partial_{t}\rho(t,y)&=\mathcal{L}[\rho](t,y)+\mathcal{R}[\rho](t,y), y\in\Theta\label{Eq:EvolutionEquation}
\end{align}
where
\[
\mathcal{L}[\rho](t,x,1)=-Au(t,x), \quad \mathcal{L}[\rho](t,x,2)=-A^{\dagger}\hat{u}(t,x),\quad \mathcal{L}[\rho](t,x,3)=0
\]
 and
 \[\mathcal{R}[\rho](t,x,1)=f(x,\theta(t)),\quad \mathcal{R}[\rho](t,x,2)=u(t,x)-h(x), \quad \mathcal{R}[\rho](t,x,3)=- \alpha(t) \bigg{(} ( \nabla_{\theta} f ( \theta(t) ), \hat u(t) ) + \gamma \theta(t) \bigg{)}.
 \]

We note that we have slightly abused notation here because  $\mathcal{L}[\rho](t,x,3)=0$. However, this notation is convenient  because it allows us to describe the PDE in question in the form (\ref{Eq:EvolutionEquation}) as  a single vector valued evolution equation.

Let us define the norm $\norm{w}_{2,T}=\sup_{t\in[0,T]} \norm{w(t)}_{L^{2}(U)}$ if $w=w(t,x)$ and $\norm{w}_{2,T}=\sup_{t\in[0,T]} \norm{w(t)}_{\ell_{2}}$ if $w=w(t)$. Here $\ell_2$ denotes the standard Euclidean norm. In particular, for $v(t,x)=(u(t,x),\hat{u}(t,x),\theta(t))$ we shall have
\[
\norm{v}_{2,T}=\norm{u}_{2,T}+\norm{\hat{u}}_{2,T}+\norm{\theta}_{2,T}.
\]
 where, the first two components depend on $(t,x)$ and the last component depends only on $t$.

We will be working with mild solutions.  Due to Assumption \ref{A:Assumption0}, the operators $A$ and $A^{\dagger}$ are generators of  analytic contraction semigroups $\{S(t)\}_{t\geq 0}$ and $\{S^{\dagger}(t)\}_{t\geq 0}$ respectively on $L^{2}(U)$.

Then, we can write for the mild solution of (\ref{Eq:EvolutionEquation}) that
\begin{align}
\rho(t,y) &= H[\rho](t,y),\nonumber
\end{align}
where
\begin{align}
H[\rho](t,y)&=\begin{cases}
H[\rho](t,x,1) \\
H[\rho](t,x,2) \\
H[\rho](t,x,3)
\end{cases}=
\begin{cases}
S(t)u_{0}(x)+\int_{0}^{t}S(t-s)f(x,\theta(s))ds\\
S^{\dagger}(t)\hat{u}_{0}(x)+\int_{0}^{t}S^{\dagger}(t-s)(u(s,x)-h(x))ds\\
- \alpha(t) \bigg{(} ( \nabla_{\theta} f ( \theta(t) ), \hat u(t) ) + \gamma \theta(t) \bigg{)}.
\end{cases}\label{Eq:MildSolution}
\end{align}

Now the properties of the  analytic contraction semigroups $\{S(t)\}_{t\geq 0}$ and $\{S^{\dagger}(t)\}_{t\geq 0}$ guarantee that there exist an increasing continuous function $D(t)$ with $\lim_{t\rightarrow 0}D(t)=0$ (possible different from line to line below) such that
\begin{align}
\norm{H[\rho](\cdot,1)}_{2,T}&\leq \norm{u_{0}}_{L^{2}}+D(T)\nonumber\\
\norm{H[\rho](\cdot,2)}_{2,T}&\leq \norm{\hat{u}_{0}}_{L^{2}}+D(T)(1+\norm{u}_{2,T})\nonumber\\
\norm{H[\rho](\cdot,3)}_{2,T}&\leq \norm{\theta_{0}}_{\ell_{2}}+D(T)(\norm{\hat{u}}_{2,T}+\norm{\theta}_{2,T})\label{Eq:NormBounds}
\end{align}
 and for $\rho=(u_{1},\hat{u}_{1},\theta_{1})$ and $q=(u_{2},\hat{u}_{2},\theta_{2})$
\begin{align}
\norm{H[\rho](\cdot,1)-H[q](\cdot,1)}_{2,T}&\leq D(T) \norm{\theta_{1}-\theta_{2}}_{2,T}\nonumber\\
\norm{H[\rho](\cdot,2)-H[q](\cdot,2)}_{2,T}&\leq D(T)\norm{u_{1}-u_{2}}_{2,T}\nonumber\\
\norm{H[\rho](\cdot,3)-H[q](\cdot,3)}_{2,T}&\leq D(T)\left(\norm{\hat{u}_{1}-\hat{u}_{2}}_{2,T}+\norm{\theta_{1}-\theta_{2}}_{2,T}(1+\norm{\hat{u}_{1}}_{2,T})\right)\label{Eq:NormLipschitz}
\end{align}

The linear growth bounds of the operator $H$  as given by (\ref{Eq:NormBounds}) together with the local Lipschitz continuity property demonstrated in
(\ref{Eq:NormLipschitz}) allow us to conclude via  the classical Picard-Lindel\"{o}f theorem for Banach valued ODE's that for every $\rho_{0}\in X^{2}$ there exists a unique local mild solution $\rho\in\mathcal{C}([0, T_0], X^{2})$  of (\ref{Eq:EvolutionEquation}) for some sufficiently small $0<T_0<\infty$.

Next we want to show that this solution can be extended globally. To do so it is enough to establish a global bound for the $\norm{\cdot}_{X^{2}}$ of solutions. Indeed, the analytic contraction semigroups $\{S(t)\}_{t\geq 0}$ and $\{S^{\dagger}(t)\}_{t\geq 0}$ guarantee that there is a constant $K<\infty$ (independent of $t$) such that
\begin{eqnarray}
\norm{u(t)}_{L^{2}} &\leq& \norm{S(t) u(0)}_{L^{2}}  + \int_0^t \norm{S(t-s) f(\theta(s)) }_{L^{2}}  ds \notag \\
&\leq& \norm{u(0)}_{L^{2}} + C \int_0^t  e^{- \lambda (t -s )}  ds \notag \\
&\leq& K.\label{Eq:EstimateForu}
\end{eqnarray}

Analogously, for a potentially different constant $K<\infty$ and using estimate (\ref{Eq:EstimateForu})
\begin{eqnarray}
\norm{\hat{u}(t)}_{L^{2}} &\leq& \norm{^{\dagger}S(t) \hat{u}(0)}_{L^{2}}  + \int_0^t \norm{S^{\dagger}(t-s) (u(s)-h) }_{L^{2}}  ds \notag \\
&\leq& \norm{\hat{u}(0)}_{L^{2}} +  \int_0^t  e^{- \lambda (t -s )}(\norm{u(s)}_{L^{2}}+\norm{h}_{L^{2}})  ds \notag\\
&\leq& K. \label{Eq:EstimateForHatu}
\end{eqnarray}

Let us next show that $\theta(t)$ is uniformly bounded in time. Define the quantity $Q(t) =  \norm{ ( \nabla_{\theta} f ( \theta(t) ), \hat u(t) ) }^{2}_2$,  and notice that due to estimate (\ref{Eq:EstimateForHatu}) and the bound on $\nabla_{\theta} f(\theta)$ by Assumption \ref{AssumptionLinear1}, we obtain that $\sup_{t \geq 0} Q(t)<C$ for some appropriate constant $C<\infty$.

By direct calculation, keeping in mind the equation that $\theta(t)$ satisfies and H\"{older} inequality, we obtain
\begin{align}
\frac{d}{dt} \norm{\theta(t)}^{2}_2 &=-2\gamma \alpha(t) \norm{\theta(t)}^{2}_2 -2\alpha(t) \theta_{t} ( \nabla_{\theta} f ( \theta(t) ), \hat u(t) )\nonumber\\
&\leq -\gamma \alpha(t)\norm{\theta(t)}^{2}_2 +\frac{4}{\gamma}\alpha(t) Q(t). \nonumber
\end{align}

By comparison principle we then have that
\begin{align}
\norm{\theta(t)}^{2}_2&\leq e^{-\gamma\int_{0}^{t}\alpha(s)ds}\norm{\theta(0)}^{2}_2
+\frac{4}{\gamma} \int_{0}^{t}\alpha(s)e^{-\gamma\int_{s}^{t}\alpha(r)dr}Q(s)ds\nonumber\\
&\leq e^{-\gamma\int_{0}^{t}\alpha(s)ds}\norm{\theta(0)}^{2}_2
+C \int_{0}^{t}\alpha(s)e^{-\gamma\int_{s}^{t}\alpha(r)dr}ds\label{Eq:EstimateForTheta}
\end{align}
and the result follows by requiring that for some $C<\infty$ (see Assumption \ref{A:LearningRateAssumption})
\[
\sup_{t\geq 0} e^{-\gamma\int_{0}^{t}\alpha(s)ds} + \sup_{t\geq 0} \int_{0}^{t}\alpha(s)e^{-\gamma\int_{s}^{t}\alpha(r)dr}ds\leq C
\]

All in all we obtain the a-priori global estimate
\begin{align}
\norm{\rho}_{2,\infty}&\leq K\nonumber
\end{align}
for some finite constant $K<\infty$. With this a-priori bound the solution can be extended indefinitely in time. Thus a unique global solution exists. This means that there exists a unique global mild solution $\rho\in\mathcal{C}([0, \infty), X^{2})$.

Essentially the same argument as above shows that if the initial data are in $L^{q}$ and $h\in L^{q}$, for $q>2$, then we will have that there is a unique global mild solution $\rho\in\mathcal{C}([0, \infty), L^{q}(\Theta,dn))$.

Let us now discuss regularity. We will prove that for initial data and $h$ in $L^{q}$, we actually have that $u,\hat{u}\in\mathcal{C}([0, \infty), L^{p})$ for any $p\in [q,\infty)$.  We will use a bootstrap argument. Due to Sobolev embedding theorem and Riesz-Thorin theorem  the following $L^{q}\rightarrow L^{p}$ estimate for the semigroup $S(t)$ holds
\begin{align}
\norm{S(t)g}_{L^{p}(U)}&\leq C (t\wedge 1)^{-\frac{n}{2}(\frac{1}{q}-\frac{1}{p})}\norm{g}_{L^{q}(U)}\label{Eq:SemigroupEstimate}
\end{align}
where $n$ is the spatial dimension, $p\geq q$ and $g$ a test function. Then, let us consider $p>q$ such that $\frac{1}{n}=\frac{1}{q}-\frac{1}{p}$ and assume that we know $\norm{u}_{L^{q}}<\infty$. Consider an initial time $t=\epsilon$ for some fixed $\epsilon>0$ and using  (\ref{Eq:SemigroupEstimate}) we have for $u$ (we use the mild formulation of the solution )
\begin{align}
\norm{u(t+\epsilon)}_{L^{p}}&\leq C\left[ t^{-\frac{1}{2}}\norm{u(\epsilon)}_{L^{q}}+\int_{0}^{t}(t-s)^{-\frac{1}{2}}\norm{u(s+\epsilon)}_{L^{q}}\right]ds\nonumber\\
&\leq C\left[ t^{-\frac{1}{2}}\norm{u(\epsilon)}_{L^{q}}+t^{\frac{1}{2}}\sup_{s\in[\epsilon,\epsilon+t]}\norm{u(s)}_{L^{q}}\right]ds. \nonumber
\end{align}

Next consider the solution $u$ starting at time $t=2\epsilon$ with initial data $u(2\epsilon)\in L^{p}$. Then, we will have that $u\in\mathcal{C}([2\epsilon,\infty), L^{p})$. Notice now that $\epsilon>0$ is arbitrary. Thus we obtain that $u\in\mathcal{C}([0,\infty),L^{p})$ for $p\geq q$ such that $\frac{1}{n}=\frac{1}{q}-\frac{1}{p}$. Using this argument inductively, first with $q=2$ and $p>2$ such that  $\frac{1}{n}=\frac{1}{2}-\frac{1}{p}$ we then get that $u\in\mathcal{C}([0,\infty), L^{p})$ for any  $p\in[2,\infty)$.

Following exactly the same process and using that $u\in\mathcal{C}([0,\infty), L^{p})$ for any  $p\in[2,\infty)$ we then obtain that if $h\in L^{p}$ for all $p\geq 2$, then $\hat{u}\in\mathcal{C}([0,\infty), L^{p})$ for any  $p\in[2,\infty)$ as well.

Next, we notice that the forcing term $\mathcal{R}$ in (\ref{Eq:EvolutionEquation}) is in $L^{p}(U)$, so by the parabolic estimates in Section IV.3 of \cite{Ladyzenskaja}, we get that $u,\hat{u}\in \mathcal{C}((0,\infty); W^{2,p}_{0}(U))\cap \mathcal{C}^{1}((0,\infty);L^{p}(U))$ and if the initial data $u_{0},\hat{u}_{0}\in W^{2,p}(U)$, then $u,\hat{u}\in \mathcal{C}([0,\infty); W^{2,p}_{0}(U))\cap \mathcal{C}^{1}([0,\infty);L^{p}(U))$ for any $p\geq 2$. This concludes the proof of the theorem.

\end{proof}

\section{Convergence to a stationary point}\label{S:ConvergenceStationaryPoint}

The main result of this section is the convergence result for $\theta(t)$. It says that as $t\rightarrow\infty$, $\theta(t)$ converges to a stationary point of the cost function $J(\theta)$.
\begin{theorem} \label{ConvergenceTheorem}
Assume Assumptions \ref{A:Assumption0}, \ref{A:LearningRateAssumption}, \ref{AssumptionLinear1} and \ref{AssumptionLinear2}. Then, we have that
\begin{eqnarray}
\lim_{t \rightarrow \infty} \norm{\nabla J( \theta(t))}_2 = 0.\nonumber
\end{eqnarray}
\end{theorem}

The proof of Theorem \ref{ConvergenceTheorem} is be a consequence of series of lemmas. In Subsection \ref{SS:DecayRates} we establish necessary decay rates for the solution to (\ref{TimeRIntro}). These results are then used in Subsection \ref{SS:ProofConv} to characterize the behavior of $\theta(t)$ for large times and eventually prove Theorem \ref{ConvergenceTheorem}.

\subsection{ Decay rates for the online adjoint algorithm (\ref{TimeRIntro})}\label{SS:DecayRates}

In this subsection, we establish some necessary decay rates for the online adjoint algorithm (\ref{TimeRIntro}).

First, Lemma \ref{L:DerivativeBound1} has a critical bound on $ \norm{\frac{\partial u^{\ast}}{\partial t } }_{H}$ and on $ \norm{\frac{\partial \hat{u}^{\ast}}{\partial t } }_{H}$.
\begin{lemma}\label{L:DerivativeBound1}
Under Assumptions \ref{A:Assumption0} and \ref{AssumptionLinear1}, there exists a constant $C<\infty$ such that
\[
 \norm{\frac{\partial u^{\ast}}{\partial t } }_{H} + \norm{\frac{\partial \hat{u}^{\ast}}{\partial t } }_{H}< C \alpha(t)
 \]
 and consequently $\displaystyle \lim_{t \rightarrow \infty} \norm{ \frac{\partial u^{\ast}}{\partial t } }_{H}= \displaystyle \lim_{t \rightarrow \infty} \norm{ \frac{\partial \hat{u}^{\ast}}{\partial t } }_{H}=0$.
\end{lemma}
\begin{proof}
 First, we study $\norm{ \frac{\partial u^{\ast}}{\partial t } }_{H}$. We see that $\frac{\partial u^{\ast}}{\partial t}(t,x)$ satisfies the PDE
\begin{align}
A\frac{\partial u^{\ast}}{\partial t }(t,x) &=  \nabla_{\theta} f(x, \theta(t) )^{\top} \frac{d \theta}{dt} = -\alpha(t) \nabla_{\theta} f(x, \theta(t) )^{\top} \bigg{(} \left( \nabla_{\theta} f (\theta(t) ), \hat u(t) \right) + \gamma \theta(t) \bigg{)},\quad x\in U, t>0 \nonumber\\
\frac{\partial u^{\ast}}{\partial t }&=0, \quad x\in\partial U, t>0
\label{UastPDE}
\end{align}

By the coercivity assumption on $A$ by Assumption \ref{A:Assumption0} we subsequently obtain
\begin{align}
\norm{ \frac{\partial u^{\ast}}{\partial t }}^{2}_{H} &\leq \frac{1}{\lambda} \left( A \frac{\partial u^{\ast}}{\partial t }, \frac{\partial u^{\ast}}{\partial t } \right) \nonumber\\
&\leq \frac{1}{\lambda} \alpha(t) \left|\left( \nabla_{\theta} f( \theta(t) )^{\top} \bigg{(} \left( \nabla_{\theta} f (\theta(t) ), \hat u(t) \right) + \gamma \theta(t) \bigg{)}, \frac{\partial u^{\ast}}{\partial t } \right) \right|\nonumber\\
&\leq \frac{1}{\lambda} \alpha(t) \norm{ \nabla_{\theta} f( \theta(t) )^{\top} \bigg{(} \left( \nabla_{\theta} f (\theta(t) ), \hat u(t) \right) + \gamma \theta(t) \bigg{)}}_{H}\norm{ \frac{\partial u^{\ast}}{\partial t } }_{H}\nonumber
\end{align}
and the result follows directly by Assumption \ref{AssumptionLinear1} on $\nabla_{\theta}f$ and estimate (\ref{UnitormThetaBound}).

Let us now turn our attention to $\norm{ \frac{\partial \hat{u}^{\ast}}{\partial t } }_{H}$. By differentiation, we obtain that $\frac{\partial u^{\ast}}{\partial t }(t,x) $ satisfies the PDE
\begin{align}
A^{\dagger}\frac{\partial \hat{u}^{\ast}}{\partial t }(t,x) &=  \frac{\partial u^{\ast}}{\partial t }(t,x),\quad x\in U, t>0 \nonumber\\
\frac{\partial \hat{u}^{\ast}}{\partial t }&=0, \quad x\in\partial U, t>0
\label{UastPDE2}
\end{align}

The result then follows by the coercivity condition on $A^{\dagger}$ by Assumption \ref{A:Assumption0} and due to the fact that $\norm{\frac{\partial u^{\ast}}{\partial t } }_{H}\leq C \alpha(t)$. This completes the proof of the lemma.
\end{proof}

Let us consider the difference
\begin{align}
\phi(t,x) = u(t,x) - u^{\ast}(t,x).\label{Eq:Phi}
\end{align}

$\phi(t)$ satisfies the PDE
\begin{align}
\frac{\partial \phi}{\partial t}(t,x) &= - A \phi(t,x) - \frac{\partial u^{\ast}}{\partial t}(t,x),\quad x\in U, t>0\nonumber\\
\phi(t,x)&=0, \quad x\in\partial U, t>0\nonumber\\
\phi(0,x)&= u_{0}(x)-u^{\ast}(0,x), \quad x\in U
\label{PhiPDE}
\end{align}

\begin{lemma} \label{PhiLemma}
Under Assumptions \ref{A:Assumption0}, \ref{A:LearningRateAssumption} and \ref{AssumptionLinear1} we have that
\begin{align}
\lim_{t \rightarrow \infty} \norm{ \phi(t) }_{H} &= 0,
\label{PhiLimit}
\end{align}
and there is some finite $T^{*}<\infty$ such that for all $t\geq T^{*}$
\begin{align}\label{Eq:DecayRatePhi}
\norm{ \phi(t) }_{H}&\leq C\left(e^{-\lambda t}+ \alpha(t)\right),
\end{align}
where $C<\infty$ is an unimportant constant.
\end{lemma}
\begin{proof}
We begin by proving that $\frac{\partial u^{\ast}}{\partial t}(t)$ is globally Lipschitz in time. Differentiating the elliptic PDE that $u^{\ast}$ satisfies twice with respect to $t$ yields for  $x\in U$ and $t\geq 0$
\begin{eqnarray}
A\frac{\partial^2 u^{\ast}}{\partial t^2 } &=& -\alpha'(t) \nabla_{\theta} f( \theta(t) )^{\top} \bigg{(} \left( \nabla_{\theta} f ( \theta(t) ), \hat u(t) \right) + \gamma \theta(t) \bigg{)} \notag \\
&-& \alpha(t)  \frac{\partial}{\partial t} \bigg{[} \nabla_{\theta} f( \theta(t) )^{\top} \bigg{(} \left( \nabla_{\theta} f ( \theta(t) ), \hat u(t) \right) + \gamma \theta(t) \bigg{)} \bigg{]}.\nonumber
\end{eqnarray}
Therefore, using the bounds on $f(\theta), \nabla f(\theta), \theta(t), u(t),$ and $\hat u(t)$, we can show, using the coercivity Assumption \ref{A:Assumption0} and the Cauchy-Schwarz inequality as in Lemma \ref{L:DerivativeBound1},  that $\displaystyle \sup_{t<\infty} \norm{\frac{\partial^2 u^{\ast}}{\partial t^2 }}_{H}\leq C$.    We provide the necessary calculations below for completeness.

\begin{eqnarray}
\norm{ \frac{\partial^2 u^{\ast}}{\partial t^2 } }^2_H &\leq& \frac{1}{\lambda} ( A\frac{\partial^2 u^{\ast}}{\partial t^2 } , \frac{\partial^2 u^{\ast}}{\partial t^2 }  ) \leq \frac{1}{\lambda}  \norm{  A\frac{\partial^2 u^{\ast}}{\partial t^2 } }_H \norm{ \frac{\partial^2 u^{\ast}}{\partial t^2 } }_H \notag \\
&\leq& \frac{K}{\lambda}  \norm{ \frac{\partial^2 u^{\ast}}{\partial t^2 } }_H, \nonumber
\end{eqnarray}
where $K$ is a constant. Re-arranging yields, for $t \geq 0$,
\begin{eqnarray}
\norm{ \frac{\partial^2 u^{\ast}}{\partial t^2 } }_H \leq C, \nonumber
\end{eqnarray}
where $C$ is a constant. This, then gives
\begin{align}
\norm{ \frac{\partial u^{\ast}}{\partial t }(t) - \frac{\partial u^{\ast}}{\partial t }(s) }_{H} &= \norm{ \int_s^t \frac{\partial^2 u^{\ast}(\rho)}{\partial \rho^2 } d\rho  }_{H} \leq \int_s^t \norm{ \frac{\partial^2 u^{\ast}(\rho)}{\partial \rho^2 } }_H d\rho  \leq C |t - s |.
\end{align}

Therefore, we can write
\begin{eqnarray}
\phi(t) = S(t) \phi(0) - \int_0^t S(t-\tau) \frac{\partial u^{\ast}}{\partial \tau} d \tau, \nonumber
\end{eqnarray}
where $\norm{S(t)}_{H} \leq  e^{- \lambda t }$ with $\lambda > 0$ by Assumptions \ref{A:Assumption0}.

Due to Lemma \ref{L:DerivativeBound1}, for any $\epsilon > 0$, there exists a $s$ such that $\norm{ \frac{\partial u^{\ast}}{\partial \tau} }_H < \epsilon$ for $t > s$. By the triangle inequality,
\begin{eqnarray}
\norm{\phi(t) }_H &\leq&  \norm{S(t)\phi(0)}_H  + \int_0^t \norm{ S(t-\tau) \frac{\partial u^{\ast}}{\partial \tau}}_H  d \tau \notag \\
&\leq& e^{- \lambda t } \norm{ \phi(0) }_H  + C_2 \int_0^t e^{- \lambda (t - \tau) } \alpha(\tau) d \tau. \notag
\end{eqnarray}

Let us now define $I(t)=\int_0^t e^{- \lambda (t - \tau) } \alpha(\tau) d \tau$. Next, it is easy to show that $\lim_{t\rightarrow\infty}I(t)=0$ and in particular that the integral term goes to zero at the rate of $\alpha(t)$ in the sense that $\lim_{t\rightarrow\infty}\frac{I(t)}{\alpha(t)}=\frac{1}{\lambda}$. These observations imply that there is a finite $T^{*}<\infty$ such that for all $t\geq T^{*}$, we have that $I(t)\leq C \alpha(t)$ for some constant $C<\infty$.

Hence we indeed get that both (\ref{PhiLimit}) and (\ref{Eq:DecayRatePhi}) hold, concluding the proof of the lemma.
\end{proof}

Define $\Psi(t,x) = \hat u (t,x) - \hat u^{\ast}(t,x)$. A similar lemma can also be proven for the limit of $\Psi(t)$.
\begin{lemma}\label{L:ConvergenceRatePsi0}
Under Assumptions \ref{A:Assumption0}, \ref{A:LearningRateAssumption} and \ref{AssumptionLinear1} we have that
\begin{eqnarray}
\lim_{t \rightarrow \infty} \norm{ \Psi(t) }_{H} = 0,
\label{PsiLimit}
\end{eqnarray}
and there is some finite $T^{*}<\infty$ such that for all $t\geq T^{*}$
\begin{align}\label{Eq:DecayRatePsi}
\norm{ \Psi(t) }_{H}&\leq C\left(e^{-\lambda t}t+ \alpha(t)\right),
\end{align}
where $C<\infty$ is an unimportant constant.
\end{lemma}
\begin{proof}
$\Psi(t)$ satisfies the PDE
\begin{eqnarray}
\frac{\partial \Psi}{\partial t}(t,x) = - A^{\dagger} \Psi(t,x) +\phi(t,x) - \frac{\partial \hat u^{\ast}}{\partial t}(t,x).
\label{PsiPDE}
\end{eqnarray}

Exactly, as it was done in Lemma \ref{PhiLemma} we can show that $t\mapsto \frac{\partial \hat u^{\ast}}{\partial t}(t)$ is globally Lipschitz. Together with Assumption \ref{A:Assumption0} we write
\begin{eqnarray}
\Psi(t) = S^{\dagger}(t) \Psi(0) + \int_0^t S^{\dagger}(t-\tau) \left(\phi(\tau)-\frac{\partial \hat u^{\ast}}{\partial t}(\tau) \right)d \tau,\notag
\end{eqnarray}
where $S^{\dagger}(t)$ is the analytic contraction semigroup generated by $A^{\dagger}$ satisfying $\norm{S^{\dagger}(t)}_{H} \leq  e^{- \lambda t }$ with $\lambda > 0$. Using the same reasoning as in Lemma \ref{PhiLemma}, and the fact that $\lim_{t\rightarrow\infty}\left(\norm{\phi(t)}_{H}+\norm{\frac{\partial \hat u^{\ast}}{\partial t}(t)}_{H}\right)=0$, we can prove (\ref{PsiLimit}). The decay rates for $\norm{\phi(t)}_{H}$ and $\norm{\frac{\partial \hat u^{\ast}}{\partial t}(t)}_{H}$ from Lemma \ref{PhiLemma} and \ref{L:DerivativeBound1} respectively prove (\ref{Eq:DecayRatePsi}), the same way (\ref{Eq:DecayRatePhi}) was proven. This concludes the proof of the lemma.
\end{proof}

\subsection{Proof of Theorem \ref{ConvergenceTheorem}}\label{SS:ProofConv}

Let us now return to the equation for $\theta(t)$.
\begin{eqnarray}\label{Eq:DynamicsTheta}
\frac{d \theta}{dt} &=& - \alpha(t) \bigg{(}  \left( \nabla_{\theta} f ( \theta(t) ), \hat u(t) \right) + \gamma \theta(t) \bigg{)} \notag \\
&=& - \alpha(t) \nabla_{\theta} J(\theta(t) ) -  \alpha(t) \bigg{(} \left( \nabla_{\theta} f ( \theta(t) ), \hat u^{\ast}(t) \right) - \left(  \nabla_{\theta} f ( \theta(t) ), \hat u(t)  \right) \bigg{)} \notag \\
&=&  - \alpha(t) \nabla_{\theta} J(\theta(t) ) - \alpha(t) \left(  \nabla_{\theta} f ( \theta(t) ), \Psi(t) \right).
\end{eqnarray}

The second term on the RHS will converge to zero as $t \rightarrow \infty$ due to (\ref{PsiLimit}). This means that asymptotically $\theta$ will be updated in the direction of steepest descent. Theorem \ref{ConvergenceTheorem} rigorously proves this along with proving convergence of $\theta(t)$ to a critical point of $J(\theta)$.

The structure of the proof proceeds in a spirit similar to \cite{SirignanoSpiliopoulos2017} with certain differences that will be highlighted below as needed. For completeness, we present the whole argument with the proper adjustments.

Let $\epsilon>0$ be given and let $\mu=\mu(\epsilon)>0$ to be chosen later on. We define the following cycle of times
\begin{eqnarray*}
0=\sigma_0 \leq \tau_1 \leq \sigma_1 \leq \tau_2 \leq \sigma_2 \leq \dots
\end{eqnarray*}
where for $k=1,2,\cdots$
\begin{align*}
\tau_{k}&=\inf\left\{t>\sigma_{k-1}:  \|\nabla J(\theta(t))\|\geq \epsilon\right\},\nonumber\\
\sigma_{k}&=\sup\left\{t>\tau_{k}:  \frac{\|\nabla J(\theta(\tau_{k}))\|}{2}\leq \|\nabla J(\theta(s))\|\leq 2\|\nabla J(\theta(\tau_{k}))\| \text{ for all }s\in[\tau_{k},t] \text{ and }\int_{\tau_{k}}^{t}\alpha(s)ds\leq \mu\right\}.
\end{align*}

Essentially, the sequence of times $\{\sigma_k\}_{k\in\mathbb{N}}$ and $\{\tau_k\}_{k\in\mathbb{N}}$ keep track of the times in which $\norm{\nabla J(\theta(t))}$ is within a ball of radius $\epsilon$ and away from it.

Next, let us define the corresponding intervals of time $J_{k}=[\sigma_{k-1},\tau_{k})$ and $I_{k}=[\tau_{k},\sigma_{k})$. Clearly, when $t\in J_{k}$, then we have that $ \|\nabla J(\theta(t))\|<\epsilon$.

Let us now go back to (\ref{Eq:DynamicsTheta}) and define the integral term
\begin{align}
\Delta_{s,t}&=\int_{s}^{t}\alpha(\rho) \left(  \nabla_{\theta} f ( \theta(\rho) ), \Psi(\rho) \right)d\rho.\label{Eq:RemainderTheta}
\end{align}

We first show that $\Delta_{\tau_{k},\sigma_{k}}$ decays  to zero. In particular, we have the following lemma.
\begin{lemma}\label{L:DeltaDecay}
Assume that Assumptions \ref{A:Assumption0}, \ref{A:LearningRateAssumption} and \ref{AssumptionLinear1} hold. Let us fix some $\eta>0$. Then, we have that $\lim_{k\rightarrow\infty}\norm{\Delta_{\tau_k,\sigma_{k}+\eta}}_{2}=0$.
\end{lemma}

\begin{proof}[Proof of Lemma \ref{L:DeltaDecay}]
We notice that for some constant $C<\infty$ that may change from line to line
\begin{align}
\sup_{t>0}\norm{\Delta_{0,t}}_{2}&\leq \int_{0}^{\infty}\alpha(\rho) \norm{ \left( \nabla_{\theta} f ( \theta(\rho) ), \Psi(\rho) \right)}_{2}d\rho\nonumber\\
&\leq \int_{0}^{\infty}\alpha(\rho) \norm{  \nabla_{\theta} f ( \theta(\rho) )}_{H} \norm{\Psi(\rho) }_{H}d\rho\nonumber\\
&\leq C \int_{0}^{\infty}\alpha(\rho)  \norm{\Psi(\rho) }_{H}d\rho\nonumber\\
&\leq C \int_{0}^{\infty}\left[\alpha^{2}(\rho)  +\alpha(\rho)\rho e^{-\lambda \rho}\right]d\rho\nonumber\\
&\leq C<\infty,\nonumber
\end{align}
where the boundedness of $\norm{  \nabla_{\theta} f ( \theta )}_{H} $ together with the decay rates from Lemma \ref{L:ConvergenceRatePsi0} were used. This immediately proves that $\lim_{k\rightarrow\infty}\norm{\Delta_{\tau_k,\sigma_{k}+\eta}}_{2}=0$ concluding the proof of the lemma.
\end{proof}

\begin{lemma}\label{L:LearningRate}
Assume that Assumptions \ref{A:Assumption0}, \ref{A:LearningRateAssumption} and \ref{AssumptionLinear1} hold. Denote by  $L_{\nabla J}$ to be the Lipschitz constant of $\nabla J$. For given $\epsilon>0$, let $\mu$ be such that $3\mu+\frac{\mu}{8\epsilon}=\frac{1}{2 L_{\nabla J}}$. Then, for $k$ large enough and for $\eta>0$ small enough (potentially depending on $k$), one has $\int_{\tau_{k}}^{\sigma_{k}+\eta}\alpha(s)ds>\mu$. In addition, we also have
$\frac{\mu}{2}\leq\int_{\tau_{k}}^{\sigma_{k}}\alpha(s)ds\leq\mu$.
\end{lemma}

\begin{proof}[Proof of Lemma \ref{L:LearningRate}]
The proof proceeds by contradiction. Let us assume that $\int_{\tau_{k}}^{\sigma_{k}+\eta}\alpha(s)ds\leq \mu$ and let $\delta>0$ be such that $\delta<\mu/8$. In addition, without loss of generality, we can assume that for the given $k$, $\eta$ is so small such that for any $s\in[\tau_{k},\sigma_{k}+\eta]$ one has $\norm{\nabla J(\theta(s))}_{2}\leq 3\norm{\nabla J(\theta (\tau_{k}))}_{2}$.

Then, invoking (\ref{Eq:DynamicsTheta}) we have
\begin{align}
\norm{\theta(\sigma_{k}+\eta)-\theta(\tau_{k})}_{2}&\leq  \int_{\tau_{k}}^{\sigma_{k}+\eta} \alpha(t) \norm{\nabla_{\theta} J(\theta(t) )}_{2}dt+ \norm{\Delta_{\tau_k,\sigma_{k}+\eta}}_{2}\nonumber\\
&\leq  3\norm{\nabla J(\theta (\tau_{k}))}_{2} \mu + \norm{\Delta_{\tau_k,\sigma_{k}+\eta}}_{2}. \nonumber
\end{align}

By Lemma \ref{L:DeltaDecay} we have that for $k$ large enough, $\norm{\Delta_{\tau_k,\sigma_{k}+\eta}}_{2}\leq \delta<\mu/8$. In addition, we also have by definition that  $\frac{\epsilon}{\norm{\nabla J(\theta (\tau_{k}))}_{2}}\leq 1$. The combination of these two results gives
\begin{align}
\norm{\theta(\sigma_{k}+\eta)-\theta(\tau_{k})}&\leq   \norm{\nabla J(\theta (\tau_{k}))}_{2} \left(3 \mu +\frac{\mu}{8\epsilon}\right)\leq   \norm{\nabla J(\theta (\tau_{k}))}_{2} \frac{1}{2 L_{\nabla J}}.\nonumber
\end{align}

This means that
\begin{align}
\norm{\nabla J(\theta(\sigma_{k}+\eta))-\nabla J( \theta(\tau_{k}))}_{2}&\leq L_{\nabla J} \norm{\theta(\sigma_{k}+\eta)-\theta(\tau_{k})}_{2}\leq \frac{1}{2} \norm{\nabla J(\theta (\tau_{k}))}_{2}.\nonumber
\end{align}

Then, this would yield
\begin{align}
\frac{1}{2} \norm{\nabla J(\theta (\tau_{k}))}_{2}&\leq \norm{\nabla J(\theta (\sigma_{k}+\eta))}_{2}\leq 2  \norm{\nabla J(\theta (\tau_{k}))}_{2}.\nonumber
\end{align}

However, this is a contradiction, because that would mean that $\int_{\tau_{k}}^{\sigma_{k}+\eta}\alpha(s)ds>\mu$, since otherwise $\sigma_{k}+\eta \in[\tau_{k},\sigma_{k}]$ which cannot happen because $\eta>0$. This concludes the proof of the first part of the lemma. The proof of the second part of the lemma goes as follows.  By its own definition, we have that $\int_{\tau_{k}}^{\sigma_{k}}\alpha(s)ds\leq \mu$. Next, we show that $\int_{\tau_{k}}^{\sigma_{k}}\alpha(s)ds\geq \mu/2$. We have shown that $\int_{\tau_{k}}^{\sigma_{k}+\eta}\alpha(s)ds> \mu$. For $k$ large enough and $\eta$ small enough we can choose that $\int_{\sigma_{k}}^{\sigma_{k}+\eta}\alpha(s)ds\leq \mu/2$. The conclusion then follows. This concludes the proof of the lemma.
\end{proof}

\begin{lemma}\label{L:DecayJ}
Assume that Assumptions \ref{A:Assumption0}, \ref{A:LearningRateAssumption} and \ref{AssumptionLinear1} hold. Assume that there exists an infinite number of intervals $I_{k}=[\tau_{k},\sigma_{k})$. Then, there is a fixed $\zeta_{1}>0$ that depends on $\epsilon$ such that for $k$ large enough
\begin{align}
J(\theta_{\sigma_{k}})-J(\theta_{\tau_{k}})\leq -\zeta_{1}.\nonumber
\end{align}
\end{lemma}
\begin{proof}[Proof of Lemma \ref{L:DecayJ}]
By chain rule we have that
\begin{align}
J(\theta(\sigma_{k})) - J (\theta(\tau_{k})) &=  -  \int_{\tau_{k}}^{\sigma_{k}} \alpha(\rho) \norm{ \nabla J (\theta(\rho)) }_2 d\rho + \int_{\tau_{k}}^{\sigma_{k}}  \alpha(\rho) \nabla J (\theta(\rho))\cdot \left(  \nabla_{\theta} f ( \theta(\rho) ), \Psi(\rho) \right)d\rho\nonumber\\
&=M_{1,k}+M_{2,k}. \nonumber
\end{align}

Let us first consider $M_{1,k}=-  \int_{\tau_{k}}^{\sigma_{k}} \alpha(\rho) \norm{ \nabla J (\theta(\rho)) }_2 d\rho$.  For $\rho\in[\tau_k,\sigma_k]$ we have that  $ \frac{\norm{\nabla J(\theta (\tau_{k}))}_{2}}{2}\leq \norm{\nabla J(\theta(\rho))}_{2}\leq 2\norm{\nabla J(\theta (\tau_{k}))}_{2} $. Thus, for sufficiently large $k$, we have by Lemma \ref{L:LearningRate}
\begin{align}
M_{1,k} &\leq  - \frac{ \norm{\nabla J (\theta(\tau_k))}_{2}^{2}}{4} \int_{\tau_{k}}^{\sigma_{k}} \alpha(\rho) d\rho \leq -  \frac{ \norm{\nabla J (\theta(\tau_k))}_{2}^2 }{8} \mu.\nonumber
\end{align}

Next, we address $M_{2,k}=\int_{\tau_{k}}^{\sigma_{k}}  \alpha(\rho) \nabla J (\theta(\rho))\cdot \left(  \nabla_{\theta} f ( \theta(\rho) ), \Psi(\rho) \right)d\rho$. Let us define
\begin{align}
\hat{M}_{s,t}&=\int_{s}^{t}  \alpha(\rho) \nabla J (\theta(\rho))\cdot \left(  \nabla_{\theta} f ( \theta(\rho) ), \Psi(\rho) \right)d\rho.\nonumber
\end{align}

Clearly, we have that $M_{2,k}=\hat{M}_{\tau_{k},\sigma_{k}}$.

We claim that $\sup_{\rho\geq 0} \norm{\nabla J (\theta(\rho))}_{2}<\infty$. For this purpose, we shall use the representation of $\nabla J(\theta)$ by (\ref{GradientObj}) together with the a-priori $H$ norm bounds for $\hat{u}^{*}$, $\nabla_{\theta} f(x,\theta)$ as well as the uniform bound on $\sup_{\rho\geq 0}\norm{\theta(\rho)}$ by (\ref{UnitormThetaBound}). These imply that indeed $\sup_{\rho\geq 0} \norm{\nabla J (\theta(\rho))}_{2}<\infty$.

Then, we have for some constant $C<\infty$ that may change from line to line
\begin{align}
\sup_{t\geq 0} \hat{M}_{0,t}&\leq C\int_{0}^{\infty} \alpha(\rho) \norm{\nabla J (\theta(\rho))}_{2}\norm{\nabla_{\theta}f(\theta(\rho))}_{H}\norm{\Psi(\rho)}_{H} d\rho\nonumber\\
&\leq C\int_{0}^{\infty} \alpha(\rho) \norm{\Psi(\rho)}_{H} d\rho\nonumber\\
&\leq C\int_{0}^{\infty} \left(\alpha^{2}(\rho) +\alpha(\rho)e^{-\lambda\rho}\rho\right)d\rho\nonumber\\
&\leq C<\infty. \nonumber
\end{align}
where we used the decay rate bound from Lemma \ref{PsiLimit}. The latter, then means that $M_{2,k}\rightarrow 0$ as $k\rightarrow\infty$.

Putting the above together, we get for $k$ large enough such that $|M_{2,k}|\leq \delta<\frac{\mu }{16}\epsilon^{2}$
\begin{align}
J(\theta(\sigma_{k})) - J (\theta(\tau_{k})) &\leq  -  \frac{ \norm{\nabla J (\theta(\tau_k))}_{2}^2 }{8} \mu+\delta\nonumber\\
&\leq -\frac{\mu}{8}\epsilon^{2}+ \frac{\mu }{16}\epsilon^{2}=-\frac{\mu }{8}\epsilon^{2}.\nonumber
\end{align}

Setting $\zeta_1=\frac{\mu}{8}\epsilon^{2}$ we conclude the proof of the lemma.
\end{proof}

\begin{lemma}\label{L:DecayJ2}
Assume that Assumptions \ref{A:Assumption0}, \ref{A:LearningRateAssumption} and \ref{AssumptionLinear1} hold. Assume that there exists an infinite number of intervals $I_{k}=[\tau_{k},\sigma_{k})$. Then, there is a fixed $0<\zeta_{2}<\zeta_{1}$  such that for $k$ large enough
\begin{align}
J(\theta_{\tau_{k}})-J(\theta_{\sigma_{k-1}})\leq \zeta_{2}.\nonumber
\end{align}
\end{lemma}
\begin{proof}[Proof of Lemma \ref{L:DecayJ2}]
Recall that $\norm{\nabla J(\theta(t))}_{2}\leq \epsilon$ for $t\in J_{k}=[\sigma_{k-1},\tau_{k}]$. By chain rule we have
\begin{align}
J(\theta(\tau_{k})) - J (\theta(\sigma_{k-1})) &=  -  \int_{\sigma_{k-1}}^{\tau_{k}} \alpha(\rho) \norm{ \nabla J (\theta(\rho)) }_2 d\rho + \int_{\sigma_{k-1}}^{\tau_{k}}  \alpha(\rho) \nabla J (\theta(\rho))\cdot \left(  \nabla_{\theta} f ( \theta(\rho) ), \Psi(\rho) \right)d\rho\nonumber\\
&\leq \int_{\sigma_{k-1}}^{\tau_{k}}  \alpha(\rho) \nabla J (\theta(\rho))\cdot \left(  \nabla_{\theta} f ( \theta(\rho) ), \Psi(\rho) \right)d\rho.\nonumber
\end{align}

As in the proof of Lemma \ref{L:DecayJ} we get that for $k$ large enough, the right hand side of the last display can be arbitrarily small. This concludes the proof of the lemma.
\end{proof}

Now we can conclude the proof of Theorem \ref{ConvergenceTheorem}.
\begin{proof}[Proof of Theorem \ref{ConvergenceTheorem}]
Fix an $\epsilon>0$. If there are finitely number of $\tau_{k}$, then there is a finite $T^{*}$ such that $\norm{\nabla J(\theta(t))}_{2}<\epsilon$ for $t\geq T^{*}$, which proves the theorem. So, we basically need to prove that there can only be finitely many $\tau_{k}$. So, let us assume that there are infinitely many instances of $\tau_{k}$. By Lemmas \ref{L:DecayJ} and \ref{L:DecayJ2} we have for sufficiently large $k$ that
\begin{align}
J(\theta_{\sigma_{k}})-J(\theta_{\tau_{k}})\leq -\zeta_{1}\nonumber\\
J(\theta_{\tau_{k}})-J(\theta_{\sigma_{k-1}})\leq \zeta_{2}\nonumber
\end{align}
with $0<\zeta_2<\zeta_1$. Let $N$ large enough so that the above relations hold simultaneously. Then we have
\begin{eqnarray*}
J(\theta(\tau_{n+1})) - J (\theta(\tau_{N})) = \sum_{k = N}^n \bigg{[} J(\theta(\sigma_{k})) - J (\theta(\tau_{k})) + J(\theta(\tau_{k+1})) - J (\theta(\sigma_{k})) \bigg{]} \leq  \sum_{k = N}^n ( - \zeta_1 + \zeta_2 )<0 .\nonumber
\end{eqnarray*}

Letting $n \rightarrow \infty$, we get that $J(\theta(\tau_{n+1})) \rightarrow - \infty$, which is a contradiction, since by definition $J(\theta) \geq 0$.  Thus, there must be at most  finitely many  $\tau_k$. This concludes the proof of the theorem.
\end{proof}
\section{Convergence rate in the strongly convex case}\label{S:ConvergenceRate}
We will now prove a convergence rate for $\theta(t)$. First, we need to strengthen the assumptions on the cost function $J(\cdot)$. Namely, we now assume that $J(\theta)$ is strongly convex.
\begin{assumption} \label{AssumptionLinear2b}
We assume the following conditions:
\begin{itemize}
\item $J(\cdot) \in C^2$.
\item There exists a unique global minimum $\theta^{\ast}$ where $\nabla_{\theta} J(\theta^{\ast}) = 0$. \item $H(\theta) =\nabla_{\theta \theta} J(\theta)$ is globally Lipschitz and $H(\theta^{\ast})$ is positive definite. That is, there exists a constant $q > 0$ such that
\begin{eqnarray}
\xi^{\top} H (\theta^{\ast})\xi \geq q \norm{\xi}_{2}^2.
\end{eqnarray}
\item The learning rate satisfies $\alpha(t) =C_{\alpha} a(t)$ where the learning rate magnitude $C_{\alpha}$ is selected such that $C_{\alpha} q > 1$. The learning rate function $a(t)$ satisfies $\displaystyle \lim_{t\rightarrow\infty} a(t)=0$ and
\begin{align}
\int_{0}^{\infty} a(s)ds&=\infty \textrm{ and } \int_{0}^{\infty} a^{2}(s)ds<\infty.\nonumber\\
\sup_{t\geq 0} \int_{0}^{t} a(s)e^{-\gamma\int_{s}^{t}a(r)dr}ds&<\infty \textrm{ and } \lim_{t\rightarrow\infty}\frac{a'(t)}{a(t)}=0.\nonumber
\end{align}
An example is $a(t) = \frac{1}{1 + t}$.
\end{itemize}
\end{assumption}

In this section we will assume without loss of generality that the learning rate is $\alpha(t) = \frac{1}{1+t}$. Even though this specific choice of the learning rate is not necessary for the results to hold, it will simplify the derivation of the convergence rate. The main result of this section is Theorem \ref{T:ConvergenceRate}.
\begin{theorem}\label{T:ConvergenceRate}
Let us assume that  $\alpha(t)=1/(1+t)$. In addition, assume that Assumptions \ref{A:Assumption0}, \ref{AssumptionLinear1} and \ref{AssumptionLinear2b} hold. Then we have that there exists a time $0<t_{0}<\infty$, such that for all $t\geq t_{0}$
\begin{eqnarray}
\norm{ \theta(t) - \theta^{\ast} }_{2} \leq C t^{-1/2}.
\end{eqnarray}
\end{theorem}

The proof of this theorem will be a consequence of a series of lemmas. Let us recall the function $\phi(t,x) = u(t,x) - u^{\ast}(t,x)$ satisfying (\ref{PhiPDE})
\begin{eqnarray}
\frac{\partial \phi}{\partial t}(t,x) &= - A \phi(t,x) - \frac{\partial u^{\ast}}{\partial t}(t,x)\nonumber
%\label{TimeR2}
\end{eqnarray}
such that by Lemma \ref{L:DerivativeBound1}, $t\mapsto \frac{\partial u^{\ast}}{\partial t}$ is globally Lipschitz and $\norm{\frac{\partial u^{\ast}}{\partial t}}_{H}\leq C\frac{1}{1+t}<C\frac{1}{t}$.

Then, we have the following lemma.
\begin{lemma}\label{L:Y_estimate}
Consider the setting of Theorem \ref{T:ConvergenceRate}.  We have that there is a finite constant $C<\infty$ such that
\begin{align}
\limsup_{t \rightarrow \infty} t^2 \left( \phi(t), \phi(t) \right) &
\leq C.
\label{LimSup1}
\end{align}
In addition,  there exists a $t_0 >0$ such that for all $t \geq t_0$ and any $0 < p < 1$,
\begin{eqnarray*}
\left( \phi(t), \phi(t) \right) \leq K t^{-2p}.
\end{eqnarray*}
\end{lemma}
\begin{proof}[Proof of Lemma \ref{L:Y_estimate}]
For notational convenience we set below $Y(t) = \left( \phi(t), \phi(t) \right)$. First we calculate
\begin{eqnarray*}
\frac{d Y}{d t} (t)= - 2 \left( \phi(t), A \phi(t) \right) - 2 \left( \phi(t), \frac{\partial u^{\ast}}{\partial t}(t) \right).
\end{eqnarray*}

For $\epsilon>0$ to be chosen later on and using the coercivity Assumption \ref{A:Assumption0},  we then have the inequality (omitting the argument $t$ for notational convenience)
\begin{eqnarray}
\frac{d Y}{d t} &\leq& - 2 \left( \phi, A \phi \right) + 2 \left|\left( \phi \epsilon,  \frac{1}{\epsilon} \frac{\partial u^{\ast}}{\partial t}(t) \right)\right| \notag \\
&\leq&  - 2 \lambda ( \phi, \phi ) +  \frac{1}{2} \epsilon^2 ( \phi, \phi ) + \frac{1}{2\epsilon^{2}} \left( \frac{\partial u^{\ast}}{\partial t}, \frac{\partial u^{\ast}}{\partial t} \right) \notag \\
&\leq& -2 (\lambda - \frac{\epsilon^2}{2}  )Y +  \frac{C}{2\epsilon^{2}}  t^{-2} \notag \\
&=& -2 b_{\epsilon} Y + C_{\epsilon} t^{-2},\notag
\end{eqnarray}
where we have used Young's inequality. The constant $b_{\epsilon} = \lambda - \frac{\epsilon^2}{2}  $ and $C_{\epsilon} = \frac{C}{2\epsilon^{2}}$. We can select $\epsilon$ such that $b_{\epsilon} > 0$.

Denoting now for notational convenience $b=b^{\epsilon}$ and with some abuse of notation setting $C=C_{\epsilon}$, let's construct the ODE
\begin{eqnarray}
\frac{d v }{d t} &=& - 2 b v + C t^{-2}, \notag \\
v(1) &=& Y(1).\notag
\end{eqnarray}

Define $\xi = Y - v$. Then, we have that $\xi(1) = 0$ and for $t \geq 1$,
\begin{eqnarray}
\frac{d \xi}{dt} &=& \frac{d Y}{dt} - \frac{dv }{dt}  \notag \\
&\leq& -2 b Y + C t^{-2} - \bigg{(} - 2 b v + C t^{-2} \bigg{)} \notag \\
&=& - 2 b ( Y -   v ) \notag \\
%&\leq& - 2b ( Y -   v )  \notag \\
&=& - 2b \xi. \notag
%\xi(1) &=& 0.\notag
\end{eqnarray}

By Gronwall's inequality $\xi \leq 0$
%\begin{eqnarray}
%\xi \leq 0,\notag
%\end{eqnarray}
 and therefore $Y \leq v$. If we can establish a convergence rate for $v$, we then have a convergence rate for $Y$.

The solution $v$ is
\begin{eqnarray*}
v(t) = e^{-2 b t} \int_1^t e^{2 b s} s^{-2} C ds,
\end{eqnarray*}

We know that
\begin{eqnarray*}
\lim_{t \rightarrow \infty} t^2 v(t)= \frac{1}{2 b}.
\end{eqnarray*}

Therefore, for a finite constant $C<\infty$ we have that
\begin{align}
\limsup_{t \rightarrow \infty} t^2 Y &\leq \limsup_{t \rightarrow \infty} t^2 v \leq C.\notag
%\label{LimSup1}
\end{align}

We also consequently know that there exists a $t_0 >0$ such that for all $t \geq t_0$ and any $0 < p < 1$,
\begin{eqnarray*}
| v(t) | \leq K t^{-2p}.
\end{eqnarray*}

Therefore, for all $t \geq t_0$,
\begin{eqnarray*}
Y(t) \leq v(t) \leq K t^{-2p}.
\end{eqnarray*}
concluding the proof of the lemma.
\end{proof}

Let us now recall that $\Psi(t,x) = \hat u (t,x) - \hat u^{\ast}(t,x)$. Next, let's prove a convergence rate for $\Psi$.

\begin{lemma}\label{L:ConvergenceRatePsi}
Consider the setting of Theorem \ref{T:ConvergenceRate}. Then, we have that
\begin{eqnarray}
\limsup_{t \rightarrow \infty} t^2 \left( \Psi(t), \Psi(t) \right) \leq C,
\label{LimSupPsi}
\end{eqnarray}
for some finite constant $C<\infty$.
\end{lemma}
\begin{proof}[Proof of Lemma \ref{L:ConvergenceRatePsi}]
We first calculate that
\begin{eqnarray}
\frac{\partial \Psi(t) }{\partial t }(t,x) &=& - A^{\dagger} \Psi(t,x) + \phi(t,x) - \frac{\partial \hat u^{\ast}}{\partial t}(t,x). \nonumber
\end{eqnarray}

Define $W(t) = t^2 \left( \Psi(t), \Psi(t) \right)$. Then, omitting for notational convenience the time argument, we have
\begin{eqnarray}
\frac{d W}{dt} &=& 2t \left( \Psi, \Psi \right)+ 2 t^2 \left( \Psi, \frac{\partial \Psi}{\partial t} \right) \notag \\
&=& \frac{2}{t} W - 2 t^2 \left( \Psi, A^{\dagger} \Psi \right)  + 2 t^2 \left( \Psi, \phi \right) - 2 t^2 \left( \Psi, \frac{\partial \hat u^{\ast}}{\partial t} \right) \notag \\
&\leq& \frac{2}{t} W - 2 \lambda W + 2 t^2 \left( \Psi, \phi \right) - 2 t^2 \left( \Psi, \frac{\partial \hat u^{\ast}}{\partial t} \right),\nonumber
\end{eqnarray}
where we used the assumed coercivity of $A^{\dagger}$ (consequence of Assumption \ref{A:Assumption0}).
Let's select a $t_0 > 1$ such that $t_0^{-1} < \delta \ll c$. Then, for $t \geq t_0$ and for a constant $C<\infty$ that may change from line to line,
\begin{eqnarray}
\frac{d W}{dt} &\leq& - 2 (\lambda- \delta) W + 2 t^2 \left( \Psi, \phi \right) + 2 t^2 \left|\left( \Psi, \frac{\partial \hat u^{\ast}}{\partial t} \right)\right| \notag \\
&\leq& - 2 (\lambda- \delta) W + C \epsilon^2 t^2 \left( \Psi, \Psi \right)  + C \epsilon^{-2} t^2 \left( \phi, \phi \right) + C \epsilon^{-2} t^2 \left( \frac{\partial \hat u^{\ast}}{\partial t}, \frac{\partial \hat u^{\ast}}{\partial t} \right)\notag \\
&=& - b V +  C t^2 \left( \phi, \phi \right)+ C  t^2 \left( \frac{\partial \hat u^{\ast}}{\partial t}, \frac{\partial \hat u^{\ast}}{\partial t} \right).\nonumber
\end{eqnarray}
where we have chosen an $\epsilon > 0$ such that $b = \lambda - \delta - C \epsilon^2  > 0$. %The constants $C_4 = \epsilon^{-2} C_2$ and $C_5 = \epsilon^{-2} C_3$.

Let's construct the ODE
\begin{eqnarray}
\frac{d \hat{q}}{dt} &=& - b \hat{q} +  C t^2 \norm{\phi}^2_{H}+C t^2 \norm{\frac{\partial \hat u^{\ast}}{\partial t}}^2_{H}, \phantom{....} t \geq t_0, \notag \\
\hat{q}(t_0) &=& V(t_0).\nonumber
\end{eqnarray}

which then, by (\ref{LimSup1}) and Lemma \ref{L:DerivativeBound1}, satisfies
\begin{eqnarray}
\hat{q}(t) &=& e^{-b t} \hat{q}(t_0) +  C e^{-bt} \int_{t_0}^t e^{b s}  s^2 \left(\norm{\phi(s)}^2_{H} + \norm{\frac{\partial \hat u^{\ast}}{\partial t}(s)}^2_{H}\right)ds \notag \\
&\leq& e^{-b t} \hat{q}(t_0) + C e^{-bt} \int_{t_0}^t e^{b s}  ds \notag \\
%=& e^{-b t} q(t_0) + C e^{-b t} \frac{ e^{bs} }{b} \bigg{|}_{s = t_0}^t \notag \\
&\leq& C.\nonumber
\end{eqnarray}

Therefore, using the same ODE comparison principle as before, we have the bound
\begin{eqnarray}
\limsup_{t \rightarrow \infty} t^2 \left( \Psi(t), \Psi(t) \right) \leq C,\nonumber
%\label{LimSupPsi}
\end{eqnarray}
which concludes the proof of the lemma.
\end{proof}

We now present the proof of Theorem \ref{T:ConvergenceRate} on the convergence rate for $\theta$.
\begin{proof}[Proof of Theorem \ref{T:ConvergenceRate}]
Recall that $H(\theta) = \nabla_{\theta \theta} J(\theta)$ is the Hessian matrix. At the stationary point $\theta^{\ast}$, $H(\theta^{\ast})$ is positive definite, i.e. there exists some constant $q > 0$ such that $\xi^{\top} H (\theta^{\ast})\xi \geq q \norm{\xi}_{2}^2$.

Since, according to Theorem \ref{ConvergenceTheorem}, we have already proven convergence, we know that for $\theta^{\ast}$ such that $\nabla J(\theta^{\ast})=0$, we have that $\displaystyle \lim_{t \rightarrow \infty} \theta(t) = \theta^{\ast}$. The parameter updates satisfy
\begin{align}
\frac{d \theta}{dt} &=  - \alpha(t) \nabla_{\theta} J(\theta(t) ) -  \alpha(t) \left( \nabla_{\theta} f ( \theta(t) ),  \Psi(t) \right) \notag \\
&= - \alpha(t) H(\theta^{\ast}) ( \theta(t) - \theta^{\ast} ) - \alpha(t)  \nabla_{\theta} H(\bar{\theta}(t))_{j,k}  ( \theta(t) - \theta^{\ast} )_k ( \theta(t) - \theta^{\ast} )_j  - \alpha(t) \left( \nabla_{\theta} f ( \theta(t) ), \Psi(t) \right), \notag
\end{align}
where $\bar{\theta}(t) \in [ \theta(t), \theta^{\ast} ]$, $\nabla_{\theta} H(\bar{\theta}(t))_{j,k} \in \mathbb{R}^d$, and $ \nabla_{\theta} H(\bar{\theta}(t))_{j,k}  ( \theta(t) - \theta^{\ast} )_k ( \theta(t) - \theta^{\ast} )_j  = \displaystyle \sum_{k,j = 1}^d  \nabla_{\theta} H(\bar{\theta}(t))_{j,k}  ( \theta(t) - \theta^{\ast} )_k ( \theta(t) - \theta^{\ast} )_j$. $H(\theta)$ is the Hessian matrix and $H(\theta)_{j,k}$ is the $(j,k)$-th element of the matrix.

Define
\begin{eqnarray}
V(t) = \norm{ \theta(t) - \theta^{\ast} }_2^2.\nonumber
\end{eqnarray}

$V$ satisfies the ODE
\begin{align}
\frac{d V}{dt} &= 2 (\theta(t) - \theta^{\ast})^{\top} \frac{d \theta}{dt } \notag \\
&=  2 (\theta(t) - \theta^{\ast})^{\top}  \bigg{[} - \alpha(t) H(\theta^{\ast}) ( \theta(t) - \theta^{\ast} ) - \alpha(t)  \nabla_{\theta} H(\bar{\theta}(t))_{j,k}  ( \theta(t) - \theta^{\ast} )_k ( \theta(t) - \theta^{\ast} )_j   -  \alpha(t) \left( \nabla_{\theta} f ( \theta(t) ),  \Psi(t) \right) \bigg{]} \notag \\
&\leq - 2 \alpha(t) q  V(t)    -  2 \alpha(t)  C \norm{ \theta(t) - \theta^{\ast} } V(t)   - 2 \alpha(t)  (\theta(t) - \theta^{\ast})^{\top} \left(   \nabla_{\theta} f ( \theta(t) ), \Psi(t) \right).\nonumber
\end{align}

Since $\displaystyle \lim_{t \rightarrow \infty} \theta(t) = \theta^{\ast}$, there exists a $t_0$ such that for all $t \geq t_0$
\begin{eqnarray}
 q - C \norm{ \theta(t) - \theta^{\ast} }  > b_0 > 0.\nonumber
\end{eqnarray}

Therefore, for $t\geq t_{0}$ large enough, we have
\begin{align}
\frac{ d V}{dt} &\leq - \alpha(t) b_0 V +    \alpha(t) \left|(\theta(t) - \theta^{\ast})^{\top} \left(  \nabla_{\theta} f ( \theta(t) ), \Psi(t) \right) \right|\notag \\
&= - \alpha(t) b_0 V  + \epsilon \alpha(t) \left| (\theta(t) - \theta^{\ast})^{\top} \left(  \nabla_{\theta} f ( \theta(t) ), \frac{\Psi(t)}{\epsilon} \right) \right|\notag \\
&\leq - \alpha(t) b_0 V  + \epsilon^2 \alpha(t)^2 V + C\epsilon^{-2} \left( \Psi, \Psi \right),\nonumber
\end{align}
where we have used Young's inequality and the Cauchy-Schwartz inequality. Here we have set $C=\sup_{\theta\in \mathbb{R}^{d}}\norm{\nabla_{\theta} f ( \theta) }^{2}_{L^{2}(U)}<\infty$ by Assumption \ref{AssumptionLinear1}.

Let's select an $\epsilon$ small enough such that $b = b_0 - \epsilon^2 > 0$. Then,
\begin{eqnarray}
\frac{ d V}{dt} &\leq&  - \alpha(t) b V  + C \left( \Psi, \Psi \right).\nonumber
\end{eqnarray}

Let $\hat{q}$ satisfy the ODE
\begin{eqnarray}
\frac{ d \hat{q}}{dt} &=&  - \alpha(t) b \hat{q}  + C \left( \Psi, \Psi \right), \phantom{....} t \geq t_0, \notag \\
\hat{q}(t_0) &=& V(t_0).\nonumber
\end{eqnarray}

The following comparison principle holds
\begin{eqnarray}
V(t) \leq \hat{q}(t).\nonumber
\end{eqnarray}

Using an integrating factor and $\displaystyle \limsup_{t \rightarrow \infty} t^2 \left( \Psi, \Psi \right) \leq C$ by Lemma \ref{L:ConvergenceRatePsi}, we have that
\begin{eqnarray}
\hat{q}(t) &=& C_1 t^{-b} + C_2 t^{-b} \int_{t_0}^t  s^{b}  \left( \Psi(s), \Psi(s) \right) ds \notag \\
&\leq& C_1 t^{-b} + C_2 t^{-b} \int_{t_0}^t  s^{b-2} s^2 \left( \Psi(s), \Psi(s) \right)   ds \notag \\
&\leq& C_1 t^{-b} + C_2 t^{-b} \int_{t_0}^t  s^{b-2}   ds \notag \\
%&\leq& C_1 t^{-b} + C_2 t^{-b}  \frac{ s^{b-1} }{b -1}  \bigg{|}_{s = t_0}^{s = t} \notag \\
&=& C_1 t^{-b} + C_2 t^{-b}  \bigg{(}  t^{b-1} - t_0^{b-1} \bigg{)} \notag \\
&\leq& C t^{-1},\nonumber
\end{eqnarray}
where the constant $C_2$ may change from line to line and we have used the assumption $C_{\alpha} b > 1$. This concludes the convergence rate proof of $\theta(t)$.
\end{proof}

\end{document}